\documentclass[12pt]{article}
\usepackage{graphicx}
\usepackage{amsmath}
\usepackage{amssymb}
\usepackage{theorem}

  \usepackage{hyperref}

\numberwithin{equation}{section}

\newtheorem{thm}{Theorem}[section]
\newtheorem{lemma}[thm]{Lemma}
\newtheorem{prop}[thm]{Proposition}
\newtheorem{cor}[thm]{Corollary}
{\theorembodyfont{\rmfamily}

\newtheorem{rmk}[thm]{Remark}
}

\newcommand{\qed}{\hfill \mbox{\raggedright \rule{.07in}{.1in}}}
 
\newenvironment{proof}{\vspace{1ex}\noindent{\bf
Proof}\hspace{0.5em}}{\hfill\qed\vspace{1ex}}

\newcommand{\R}{{\mathbb R}}
\newcommand{\eps}{\epsilon}
\newcommand{\Int}{\operatorname{Int}}
\newcommand{\rank}{\operatorname{rank}}
\newcommand{\diag}{\operatorname{diag}}

\title{Time-reversibility and nonvanishing L\'evy area}

\author{Georg Gottwald \thanks{School of Mathematics and Statistics, University of Sydney, Sydney 2006 NSW, Australia}
\and 
Ian Melbourne \thanks{Mathematics Institute, University of Warwick, Coventry, CV4 7AL, UK}
}

\date{20 October, 2022. Updated 10 August, 2023}

\begin{document}

\maketitle

 \begin{abstract}
We give a complete description and clarification of the structure of the L\'evy area correction to It\^o/Stratonovich stochastic integrals arising as limits of time-reversible deterministic dynamical systems. In particular, we show that time-reversibility forces the L\'evy area to vanish only in very specific situations that are easily classified.
 \end{abstract}

\section{Introduction} 
The classical Wong-Zakai question~\cite{WongZakai65} from 1965 is concerned with weak convergence of smooth processes $W_n$ to Brownian motion $W$ and the consequences for the interpretation of the corresponding stochastic integral 
$\int W\,dW$.

In simple situations~\cite{WongZakai65,GM13b}, the limiting stochastic integrals are Stratonovich, denoted  $\int W\circ dW$, but numerous counterexamples exist in higher dimensions~\cite{McShane72,Sussmann91}. In general, there is a correction, reminiscent of the It\^{o}-Stratonovich correction, given by a deterministic quantity known as the L\'evy area $E$, whereby $\int_0^t W_n\,dW_n$ converges weakly to $\int_0^t W\circ dW+ Et$ as $n\to\infty$. Many references give closed-form formulas for the L\'evy area $E$, see~\cite{BassEtAl02,CFKMZ22,DelongEtAl14,FrizOberhauser09,KM16,Lejay05,Lim21,LopusanschiSimon18,MacKay10}.

It is by now well-understood (if not always well-known) that the L\'evy area is an important and nontrivial correction to the Wong-Zakai question.
However, there are few investigations of whether this formula for $E$ as displayed in Section~\ref{sec:setup} (which looks like it may be nonzero) is actually nonzero in the presence of additional constraints such as time-reversibility. 

The L\'evy area is skew-symmetric ($E^T=-E$) and hence vanishes in the scalar case. In higher dimensions, L\'evy area corrections vanish as a consequence of exactness, or commutativity of the defining vector fields, but such conditions are atypical outside the scalar case.
On the other hand, an
example where it is proved that $E\neq0$ is given by Hairer, Pavliotis \& Stuart, see~\cite[Section~11.7.7]{PavliotisStuart}.

From time to time, we have been asked whether time-reversibility can force ${E=0}$. (The example of Hairer {\em et al.}\ in~\cite{PavliotisStuart} is not time-reversible.) In dispersing billiard examples considered by Chernov \& Dolgopyat~\cite{ChernovDolgopyat09}, the L\'evy area is indeed zero as a consequence of time-reversibility and the structure of the equations. For Markov processes, there is a related condition,
\emph{detailed balance}, which forces $E=0$, see for example~\cite[Section~3.3.2]{JiangQianQian} or~\cite[Section~5.1]{Pavliotis14}. 
See also~\cite[Section~1.4.2]{LopusanschiSimon18} for further comments on $E$ being zero in certain time-reversible situations.

The recent work~\cite{DiamantakisHolmPavliotisArxiv} gives numerical evidence that $E\neq0$ in certain examples, but it is written (at the end of Section~4 of that paper) that $E$ could be negligible in many situations.
On the other hand, it has recently been conjectured that in numerical simulations of certain stochastic systems which were obtained via some stochastic parametrisation it is numerically advantageous to neglect the L\'evy area and set $E=0$ \cite{Fiorini23,BoulvardMemin23}. 

In this paper, we offer a complete description and clarification of the structure of the L\'evy area.
In particular, we classify the constraints on $E$ imposed by time-reversibility. The cases where $E$ is forced to vanish are easily described. 
Outside of these rare situations, we find that $E$ is typically far from negligible.

\begin{rmk}
Since our conclusions are contrary to those for systems satisfying
detailed balance,
our results can be viewed as advising caution on when to invoke detailed balance as a modelling assumption.
\end{rmk}

The remainder of this paper is organised as follows.
Section~\ref{sec:setup} contains the setup in this paper and gives a rough description of our main result, namely Theorem~\ref{thm:E0}.
In Section~\ref{sec:TR}, we introduce time-reversibility and determine the constraints imposed by the time-reversibility on the L\'evy area.
In Section~\ref{sec:E0}, we state and prove Theorem~\ref{thm:E0}, showing that there are no further constraints on the L\'evy area beyond those determined in Section~\ref{sec:TR}.
In Section~\ref{sec:surj}, we obtain a local version of this result: namely that there are no constraints on the L\'evy area under small perturbations of the fast dynamics.
Finally, in Section~\ref{sec:a}, we verify that, in reasonable situations, nonvanishing L\'evy area leads as anticipated to nontrivial corrections to limiting stochastic integrals.

\section{The setup}
\label{sec:setup}

We consider fast-slow ordinary differential equations (ODEs) on $\R^d\times\R^m$ of the form
\begin{align} \label{eq:fs} \nonumber
\dot x & = a(x)+\eps^{-1}b(x)v(y), \quad x(0)=\xi\in\R^d \\
\dot y & = \eps^{-2}g(y), \qquad \qquad\qquad y(0)\in(\Lambda,\mu)
\end{align}
where $\Lambda\subset\R^m$ is a compact subset and $\mu$ is a Borel probability measure on $\Lambda$.
The functions 
\[
a:\R^d\to\R^d, \quad b:\R^d\to\R^{d\times d}, \quad v:\R^m\to\R^d, \quad g:\R^m\to\R^m,
\]
 are assumed to be $C^r$ (for some $r\ge1$)
with $\int_\Lambda v\,d\mu=0$.
Let $g_t:\Lambda\to\Lambda$ denote the flow on $\Lambda$ generated by the ODE $\dot y = g(y)$.
We assume that $\mu$ is $g_t$-invariant, ergodic and mixing.

The aim of homogenisation is to establish, as $\eps\to0$ in~\eqref{eq:fs}, a limiting stochastic differential equation (SDE) of the form
\begin{equation} \label{eq:SDE}
dX= \tilde a(X)\,dt +b(X)\circ dW, \quad X(0)=\xi ,
\end{equation}
such that $x=x^{(\eps)}$ converges weakly to $X$. 
Here, $W$ is $d$-dimensional Brownian motion, the stochastic integral $b(X)\circ dW$ has the Stratonovich interpretation, and $\tilde a$ is a modified drift term incorporating the correction (if any) to the Stratonovich integral.

In the deterministic setting of~\eqref{eq:fs}, convergence to an SDE of the form~\eqref{eq:SDE} was obtained in~\cite{KM16} under suitable chaoticity assumptions (subsequently optimised in~\cite{CFKMZ22,Fleming22,KKM22}) on the fast dynamics $\dot y = g(y)$.
Under these assumptions, we have convergent series of Green-Kubo-type:
\begin{description}
\item[Covariance]
$\Sigma   = \int_0^\infty\int_\Lambda \big\{v\otimes(v\circ g_t)+(v\circ g_t)\otimes v\big\}\,d\mu\,dt$,
\item[L\'evy area] $E  = \int_0^\infty\int_\Lambda \big\{v\otimes(v\circ g_t)-(v\circ g_t)\otimes v\big\}\,d\mu\,dt$,
\end{description}
where $u\otimes v=uv^T\in\R^{d\times d}$ for $u,v\in\R^d$.
These series define
a positive semi-definite symmetric matrix $\Sigma\in\R^{d\times d}$ and a skew-symmetric matrix $E\in\R^{d\times d}$.
Moreover (under smoothness assumptions on $a$ and $b$ which play no further role in this paper), the solutions $x^{(\eps)}$ converge weakly to solutions of the SDE~\eqref{eq:SDE} where the Brownian motion $W$ has covariance matrix $\Sigma$ and
the modified drift term is given by
\begin{equation} \label{eq:a}
\tilde a(X)=a(X)+\tfrac12 \sum_{\alpha,\beta,\gamma=1}^d E^{\gamma\beta}\partial_\alpha b^\beta(X)b^{\alpha\gamma}(X).
\end{equation}
Here, $Z^{ij}$ denotes the $(i,j)$'th entry of a matrix $Z$ and $b^\beta$ denotes the $\beta$'th column of~$b$.

Generally, $\Sigma$ is positive definite in the setting of~\cite{KM16}. Indeed, the case $\det\Sigma=0$ is infinitely unlikely in a sense that can be made precise, see for example~\cite[Section~2.3]{CFKMZ19}.
Given its antisymmetry, a natural question is to ask whether the L\'evy area $E$ may be forced to vanish in certain circumstances. 
Clearly, if $d=1$ then $E=0$. 
In addition, if $v$ transforms as $v\circ R=v$ or $v\circ R=-v$, where $R$ is a time-reversal symmetry for the fast dynamics, then again it is easily verified (see Remark~\ref{rmk:E=0}) that $E=0$.
This situation can occur in simple physical situations where $v$ represents position or velocity, such as in the dispersing billiards examples in~\cite{ChernovDolgopyat09}.

In this paper we show that the cases $v\circ R=v$ and $v\circ R=-v$ are the only situations where time-reversal symmetry forces the L\'evy area to vanish, and typically $E\neq0$ for the remaining time-reversible systems. In particular, we show that $v$ transforms as $v\circ R=Av$ where $A=I_{d^+}\oplus(-I_{d^-})$ in appropriate coordinates (with $d^+ + d^-=d$), and that 
$E=\left(\begin{array}{cc} 0 & E_0 \\ -E_0^T & 0 \end{array}\right)$
 where $E_0$ is a general $d^+\times d^-$ matrix.

\section{Time-reversal symmetry}
\label{sec:TR}

In this section, we introduce time-reversal symmetry into the fast-slow ODE~\eqref{eq:fs} and derive simplified formulas for the covariance $\Sigma$ and L\'evy area $E$. Our only assumption in this section is that the Green-Kubo-type formulas for $\Sigma$ and $E$ converge.

We assume that there is a time-reversal symmetry $(x,y)\mapsto (Sx,Ry)$
where $S\in\R^{d\times d}$ and $R:\Lambda\to\Lambda$ satisfy
$S^2=I$ and $R\circ R=I$. 
As usual, this means that $(Sx(-t),Ry(-t))$ is a solution of~\eqref{eq:fs} whenever $(x(t),y(t))$ is a solution.
We suppose also that $\mu$ is $R$-invariant.

For the fast dynamics, time-reversibility means that 
\begin{equation} \label{eq:gt}
g_t(Ry)=Rg_{-t}(y)
\quad\text{for all $y\in\Lambda$, $t\in\R$.}
\end{equation}
Equivalently, $g(Ry)=-Rg(y)$ for all $y\in\Lambda$.

\begin{prop} \label{prop:gt}
\[
\begin{split}
\Sigma &  = \int_0^\infty\int_\Lambda \big\{v\otimes(v\circ g_t)+(v\circ R)\otimes (v\circ R\circ g_t)\big\}\,d\mu\,dt, \\
E &  = \int_0^\infty\int_\Lambda \big\{v\otimes(v\circ g_t)-(v\circ R)\otimes (v\circ R\circ g_t)\big\}\,d\mu\,dt.
\end{split}
\]
\end{prop}

\begin{proof}
By invariance of $\mu$ under the fast flow $g_t$ and $R$,
and~\eqref{eq:gt},
\[
\begin{split}
\int_\Lambda (v\circ g_t)\otimes v\,d\mu & =
\int_\Lambda (v\circ R)\otimes (v\circ g_{-t}\circ R)\,d\mu
\\ & =
\int_\Lambda (v\circ R)\otimes (v\circ R \circ  g_{t})\,d\mu.
\end{split}
\] 
Substituting this into the covariance and L\'evy area formulas from Section~\ref{sec:setup} yields the result.
\end{proof}

\begin{rmk} \label{rmk:E=0}
It can be seen already from Proposition~\ref{prop:gt} that $E=0$ if either $v\circ R=v$ or $v\circ R=-v$. (This was the case in~\cite{ChernovDolgopyat09}.)
Our main result, Theorem~\ref{thm:E0} below, implies that typically $E\neq0$ outside of these cases.
\end{rmk}

For the slow dynamics, time-reversibility means that
\[
a(Sx)+\eps^{-1}b(Sx)v(Ry)=-S\big\{a(x)+\eps^{-1}b(x)v(y)\big\} 
\quad\text{for all $x\in\R^d$, $y\in\Lambda$, $\eps>0$.}
\]
This simplifies to the requirement that
\begin{equation} \label{eq:abv}
a(Sx)=-Sa(x) 
\quad\text{and}\quad
b(Sx)v(Ry)=-b(x)v(y)
\quad\text{for all $x\in\R^d$, $y\in\Lambda$.}
\end{equation}

To avoid pathologies, from now on we suppose that 
$b:\R^d\to\R^{d\times d}$ defines a nonsingular matrix in $\R^{d\times d}$ on a dense subset of $\R^d$
and that $\{v(y):y\in\Lambda\}$ spans $\R^d$ in the sense that
$\{v(y):y\in\Lambda\}$ is not contained in a proper affine subspace of $\R^d$.  
Then the second condition in~\eqref{eq:abv} simplifies further:

\begin{lemma}   \label{lem:A}
Condition~\eqref{eq:abv} holds if and only if there exists $A\in\R^{d\times d}$ with $A^2=I$ such that
\[
a(Sx)  =-Sa(x),
\qquad
b(Sx)  =-Sb(x)A,
\qquad
v(Ry)  =Av(y),
\]
for all $x\in\R^d$, $y\in\Lambda$.
\end{lemma}

\begin{proof}
It is immediate that if $a$, $b$ and $v$ satisfy these restrictions for some $A$ with $A^2=I$, then condition~\eqref{eq:abv} holds.

Conversely, suppose that condition~\eqref{eq:abv} holds.
Let $X=\{x\in\R^d:\det b(x)\neq0\}$.
%
Define
\[
A:X\to\R^{d\times d}, \qquad A(x)=-b(x)^{-1}Sb(Sx).
\]
Then
\[
v(y)=A(x_1)v(Ry) \quad\text{and}\quad
v(Ry)=A(x_2)v(y) \quad \text{for all $x_1,x_2\in X$, $y\in\Lambda$}.
\]
Hence $v(y)=A(x_1)A(x_2)v(y)$ for all $y$ and it follows from the spanning assumption on $v$ that
\[
A(x_1)A(x_2)=I \quad\text{for all $x_1,x_2\in X$.}
\]
Taking $x_1=x_2$, we obtain that $A(x)\equiv A$ is constant on $X$ with $A^2=I$.
We immediately obtain that $v\circ R=Av$.  Also $b\circ S=-SbA$ on the dense set $X\subset\R^d$ and hence on the whole of $\R^d$ by continuity of $b$.
%
%
%
%
\end{proof}

Let $\pi^{\pm}:\R^d\to\R^d$ be the projections onto the $\pm1$ eigenspaces of $A$.  Then we can write $v$ in~\eqref{eq:fs} uniquely as $v=v^++v^-$ where $v^\pm=\pi^\pm v$. Note that
$\int_\Lambda v^\pm\,d\mu=0$.

\begin{cor} \label{cor:A}
In the $(\pi^+,\pi^-)$ coordinates, 
\[
\begin{split}
\Sigma= \left(\begin{array}{cc}\Sigma^+ & 0 \\ 0 & \Sigma^-
\end{array}\right), \qquad
E = \left(\begin{array}{cc} 0 & E_0 \\ -E_0^T & 0
\end{array}\right),
\end{split}
\]
where
\[
\Sigma^\pm = 
2\int_0^\infty\int_\Lambda v^\pm\otimes(v^\pm\circ g_t)
\,d\mu\,dt, \qquad
E_0 = 
2\int_0^\infty\int_\Lambda v^+\otimes(v^-\circ g_t)
\,d\mu\,dt.
\]
\end{cor}

\begin{proof}
Since $v\circ R=Av=v^+-v^-$, we obtain
\[
\begin{split}
\Sigma &  = 2\int_0^\infty\int_\Lambda \big\{v^+\otimes(v^+\circ g_t)+
v^-\otimes(v^-\circ g_t)\big\}\,d\mu\,dt, \\
E &  = 2\int_0^\infty\int_\Lambda \big\{v^+\otimes(v^-\circ g_t)+
v^-\otimes(v^+\circ g_t)\big\}\,d\mu\,dt.
\end{split}
\]
The result follows.
\end{proof}

\section{Generality of $E_0$}
\label{sec:E0}

In this section, we show that there are no further restrictions on the L\'evy area $E$ beyond those in Corollary~\ref{cor:A}.
Recall that the vector field $g$ in~\eqref{eq:fs} is assumed to be $C^r$ for some $r\ge1$.
We suppose that the fast dynamics defined by $g$ is mixing sufficiently quickly, so that
the series for $E$ in Section~\ref{sec:setup} and hence $E_0$ in Corollary~\ref{cor:A} converge for all $C^r$ functions $v:\Lambda\to\R^d$ with $\int_\Lambda v\,d\mu=0$. 
We exclude the uninteresting case where $\Lambda$ is a fixed point. (For then $v|_\Lambda\equiv0$, so $\Sigma=E=0$ and there is no stochasticity in the limit.)

\begin{rmk} \label{rmk:dense}
In the setting of~\cite{CFKMZ22,Fleming22,KM16,KKM22}, the matrices $\Sigma$ and $E$ depend continuously on $v\in C^r$. 
\end{rmk}


We say that a function $f$ defined on $\Lambda$ is \emph{$R$-invariant} if
$f\circ R=f$.

\begin{prop} \label{prop:ex}
For any $R$-invariant  $f\in C^r(\R^m,\R^{d^+})$,
$h\in C^{r+1}(\R^m,\R^{d^-})$ with $\int_\Lambda f\,d\mu=0$,
there exists $v\in C^r(\R^m,\R^d)$ 
with $v\circ R=Av$ and $\int_\Lambda v\,d\mu=0$ such that
$E_0 = \int_\Lambda f\otimes h\,d\mu$.
\end{prop}

\begin{proof}
Recall that $g_t$ denotes the fast flow generated by the vector field $g$ in~\eqref{eq:fs}.
We define $v=v^++v^-:\Lambda\to\R^d$ by
setting 
\[
v^+  =-f, 
\qquad v^-=\nabla h\cdot g=\sum_j \frac{\partial h}{\partial y_j}g_j.
\]

Clearly, $v^+\circ R=v^+=Av^+$.
By the chain rule, $(\nabla h)_{Ry}=(\nabla h)_yR$. 
This combined with the identity
$g\circ R=-Rg$ ensures that $v^-\circ R=-v^-=Av^-$.
Hence $v\circ R=Av$.
Also, $\int_\Lambda v^- \,d\mu=
\int_\Lambda v^-\circ R \,d\mu=
-\int_\Lambda v^- \,d\mu$. Hence 
$\int_\Lambda v^- \,d\mu=0$. 
But $\int_\Lambda v^+ \,d\mu=0$ by construction, so 
$\int_\Lambda v \,d\mu=0$.

Along solutions $y(t)$ to $\dot y=g(y)$, 
\[
v^-(y(t))=
\sum_j \frac{\partial h}{\partial y_j}(y(t))\, \dot y_j(t)
=\frac{d}{dt} h(y(t)).
\]
Hence
\(
\int_0^T v^-(y(t))\,dt= h(y(T))-h(y(0)).
\)
In other words,
\[
\int_0^T v^-\circ g_t\,dt = h\circ g_T-h.
\]
Since the flow is mixing and $\int_\Lambda v^+\,d\mu=0$,
\[
\begin{split}
\int_0^T \int_\Lambda  v^+ \otimes (v^-\circ g_t)\,d\mu\,dt & =
 \int_\Lambda v^+ \otimes \Big(h\circ g_T-h\Big)\,d\mu \\
& \to -\int_\Lambda v^+\otimes h\,d\mu
= \int_\Lambda f\otimes h\,d\mu
\end{split}
\]
as $T\to\infty$.
By the definition of $E_0$ in Corollary~\ref{cor:A} (up to a factor of $2$ which can be incorporated into $f$) we have proved the result.
\end{proof}

\begin{thm} \label{thm:E0}
For any $d^+\times d^-$ matrix $F$, there exists a $C^r$ function
$v:\R^m\to\R^d$ with $v\circ R=Av$ and $\int_\Lambda v\,d\mu=0$ such that $E_0=F$.
\end{thm}

\begin{proof}
We claim that it is possible to choose $v$ so that $E_0$ has full rank, namely
$\min\{d_+,d_-\}$. 
Assuming this is the case, let $L^\pm\in \R^{d^\pm\times d^\pm}$.
By the definition of $E_0$ in Corollary~\ref{cor:A},
transforming $v$ to $(L^+v^+,L^-v^-)$ changes $E_0$ to $L^+E_0(L^-)^T$.
By standard linear algebra, this results in any desired matrix in $\R^{d^+\times d^-}$.

It remains to prove the claim. The first step is to construct suitable $R$-invariant functions $f_1\in L^2(\Lambda,\R^{d^+})$ and $h_1\in L^2(\Lambda,\R^{d^-})$ with $\int_\Lambda f_1\,d\mu=0$ such 
that $\int_\Lambda f_1\otimes h_1\,d\mu$ has full rank.
The second step is to approximate $f_1$ and $h_1$ by smooth $R$-invariant functions $f:\R^m\to \R^{d^+}$ and $h:\R^m\to \R^{d^-}$ with
$\int_\Lambda f\,d\mu=0$ so that
$\int_\Lambda f\otimes h\,d\mu$ has full rank.
The third step is to apply Proposition~\ref{prop:ex}.

\vspace{1ex}
\noindent{\bf Step 1}
Since $\Lambda$ is not a fixed point, we can choose infinitely many orthonormal $R$-invariant functions
$\phi_j\in L^2(\Lambda)$
with $\int_\Lambda \phi_j\,d\mu=0$ and $\int_\Lambda\phi_i\phi_j\,d\mu=\delta_{ij}$.
Let $f_1=\sum_{i=1}^{d^-} \alpha_i\phi_i$ and 
$h_1=\sum_{j=1}^{d^-} \beta_j\phi_j$ where
$\alpha_i\in \R^{d^+}$, 
$\beta_j\in \R^{d^-}$. 
Then 
\[
\int_\Lambda f_1\otimes h_1\,d\mu
=\sum_{i,j}(\alpha_i\otimes \beta_j)\int_\Lambda \phi_i\phi_j\,d\mu
=\sum_{j=1}^{d^-}\alpha_j\otimes \beta_j.
\]
Let $F_1=\int_\Lambda f_1\otimes h_1\,d\mu$.
Taking $\beta_j$ to be the $j$'th canonical unit vector, 
$F_1$ is the $d^+\times d^-$ matrix with
columns $\alpha_1,\dots,\alpha_{d^-}$.  
In particular, $F_1$ is arbitrary and we can choose the $\alpha_i$ so that $F_1$ is a matrix of full rank.

\vspace{1ex}
\noindent{\bf Step 2}
Now choose $R$-invariant functions $\tilde\phi_j\in C^\infty(\R^m)$ with
$\int_\Lambda \tilde\phi_j\,d\mu=0$ and 
$\int_\Lambda|\phi_j-\tilde\phi_j|^2\,d\mu$ small\footnote{For instance, to approximate $\phi_1$ by a $C^\infty$ function, first
approximate $\phi_1$ in $L^2$ by a simple function $\sum_{k=1}^\ell c_k1_{E_k}$
(with $c_k\in\R$ and $E_k\subset\Lambda$ measurable). By outer regularity of the Borel probability measure $\mu$, there exist open neighbourhoods $U_k\subset\R^m$ of $E_k$ with $\mu(U_k\setminus E_k)$ small.
Hence $\sum_{k=1}^\ell c_k1_{U_k}$ is $L^2$-close to $\phi_1$.
Choose $V_k\subset\R^m$ closed such that $\overline{U_k}\subset\Int V_k$ with $\mu(V_k-U_k)$ small. By Urysohn's Lemma, there exists a continuous function $\psi_k:\R^m\to[0,1]$ supported on $V_k$ with $\psi_k|_{U_k}{\equiv1}$. 
In this way we obtain a continuous function $\sum_{k=1}^\ell c_k\psi_k$ that is $L^2$-close to $\phi_1$. 
Finally, each $\psi_k$ can be uniformly approximated by a $C^\infty$ function $\zeta_k$ resulting in a $C^\infty$ function $\tilde\phi_1=\sum_{k=1}^\ell c_k\zeta_k$ that is $L^2$-close to $\phi_1$.}
 and
define $f$, $h$ using $\tilde\phi_j$ in place of $\phi_j$ with $\alpha_i$, $\beta_j$ unchanged. 
This results in a matrix $F=\int_\Lambda f\otimes h\,d\mu$ close to $F_1$. In particular, taking the approximation close enough ensures that $F$ is still of full rank.

\vspace{1ex}
\noindent{\bf Step 3}
It follows from Proposition~\ref{prop:ex} that we can choose $v\in C^r(\R^m,\R^d)$ with $v\circ R=Av$ and $\int_\Lambda v\,d\mu=0$ so that $E_0=\int_\Lambda f\otimes h\,d\mu$. In particular, such $E_0$ has full rank, proving the claim.
\end{proof}

\section{Local surjectivity for perturbations of $E_0$}
\label{sec:surj}

Let $\chi:C^r(\R^m,\R^d)\to \R^{d^+\times d^-}$ be the mapping defining
$E_0=\chi(v)$ in Corollary~\ref{cor:A}.

In Section~\ref{sec:E0}, we showed that $E_0$ was a general matrix in the sense that $\chi$ is surjective.
As mentioned in Remark~\ref{rmk:dense}, $\chi$ is continuous in reasonable situations. In this section, we complete the picture by establishing a local surjectivity result for $\chi$; namely that if $\chi(v_0)=E_0$ and $E$ is close to $E_0$ of full rank, then there exists $v_1$ close to $v_0$ with $\chi(v_1)=E$.

The first step is to show that if $E_0=\chi(v_0)$ does not have full rank, then the rank can be increased under small perturbations.

\begin{lemma} \label{lem:surj}
Suppose that $E_0=\chi(v_0)$.
Then there exists $v$ arbitrarily $C^r$-close to $v_0$
such that $\rank\chi(v)=\min\{d^+,d^-\}$.
\end{lemma}


\begin{proof}
Suppose without loss of generality that $d^+\le d^-$.
By Theorem~\ref{thm:E0}, we can choose $v\in C^r(\R^m,\R^d)$ such that 
$\chi(v)=E$ 
where $E=\left(\begin{array}{c|c} I_{d^+} & 0  \end{array} \right)$.
Let $v_t=v_0+tv$ and define $E_t=\chi(v_t)$.
Then
$E_t=\left(\begin{array}{c|c} A_t & B_t  \end{array} \right)$
where $A_t=A_0+tA_1+t^2 I_{d^+}$ for some
matrices $A_0,A_1\in\R^{d^+\times d^+}$, 
$B_t\in\R^{d^+\times d^-}$.
Define the polynomial $p(t)=\det A_t$ of degree $2d^+$.
Note that $p(t)=t^{2d^+}(1+O(t^{-1}))$ so $p(t)\neq0$ for large $t$.
Since $p$ is a polynomial, it follows that there exists $\eps>0$ such that
$p(t)\neq0$ for all $t\in(0,\eps)$.
Hence $A_t$ is invertible, and so $\rank E_t=d^+$, for all $t\in(0,\eps)$.
\end{proof}

Next, we require a basic result from linear algebra.
\begin{prop} \label{prop:LA}
Let $E_0,E\in \R^{m\times n}$ be matrices of full rank and suppose that $E$ is close to $E_0$.
Then there exist near identity matrices 
$P\in\R^{m\times m}$,
$Q\in\R^{n\times n}$ such that
$PE_0Q^T=E$. 
\end{prop}

\begin{proof}
Suppose without loss of generality that $m\le n$.
Since $\rank E_0=m$, 
there exist invertible matrices 
$P_0\in\R^{m\times m}$,
$Q_0\in\R^{n\times n}$ such that
$P_0E_0Q_0^T=\left(\begin{array}{c|c} I_{m} & 0  \end{array} \right)$.
Then
$P_0EQ_0^T$ is close to $\left(\begin{array}{c|c} I_{m} & 0  \end{array} \right)$
and it is easily seen that there exist near identity matrices
$P_1\in\R^{m\times m}$,
$Q_1\in\R^{n\times n}$ corresponding to near identity row and column operations such that 
$(P_1P_0)E(Q_1Q_0)^T=\left(\begin{array}{c|c} I_{m} & 0  \end{array} \right)$.
Moreover,
$(P_0^{-1}P_1P_0)E(Q_0^{-1}Q_1Q_0)^T=E_0$.
Hence the result holds with 
$P=P_0^{-1}P_1^{-1}P_0$,
$Q=Q_0^{-1}Q_1^{-1}Q_0$.
\end{proof}

We can now state and prove the main result of this section.

\begin{thm} \label{thm:surj}
Suppose that $E_0=\chi(v_0)$ and that $E$ is of full rank and close to $E_0$.
Then there exists $v$ that is $C^r$-close to $v_0$ such that $\chi(v)=E$.
\end{thm}

\begin{proof}
Suppose without loss of generality that $d^+\le d^-$.
By Lemma~\ref{lem:surj}, we can suppose that $\rank E_0=d^+$.
By Proposition~\ref{prop:LA}, there exist near identity matrices 
$P\in\R^{d^+\times d^+}$,
$Q\in\R^{d^-\times d^-}$ such that
$PE_0Q^T=E$. Hence we can take
$v=(P\oplus Q)v_0$.~
\end{proof}

\section{Corrections to limiting stochastic integrals}
\label{sec:a}
In the previous sections, we gave a complete description of the L\'evy area, namely the skew symmetric matrix $E$ whose entries determine the correction
\begin{equation} \label{eq:correct}
\tfrac12\sum_{\alpha,\beta,\gamma=1}^d E^{\gamma\beta}\partial_\alpha b^\beta(X)b^{\alpha\gamma}(X)
\end{equation}
 to the drift term $a(X)$ as given in~\eqref{eq:a}.
This is not the complete story since even when $E$ is nonzero, it might be the case that the correction is forced to vanish due to the structure of $b$. Indeed, this happens when $b:\R^d\to\R^{d\times d}$ satisfies an exactness condition, namely that $b^{-1}=dh$ for some $h:\R^d\to\R^d$, see for example~\cite{GM13b}. 

In general, $b$ satisfies the time-reversibility constraint
\begin{equation} \label{eq:b}
b(Sx)=-Sb(x)A,
\end{equation} 
from Lemma~\ref{lem:A} which places restrictions on the correction~\eqref{eq:correct}. 

In this section, we consider some simple time-reversible situations with $E$ nonvanishing and show that in these cases the correction~\eqref{eq:correct} is typically nonzero. Given the bilinearity of this term, it suffices to exhibit a single $b$ for which the term is nonzero.

Recall that $A^2=S^2=I$ and that we can choose coordinates such that
$A=I_{d^+}\oplus I_{d^-}$. Recall also 
that $E^{\gamma\beta}=0$ if $1\le \beta,\gamma\le d^+$ and if $d^+ +1\le \beta,\gamma\le d$, while the remaining entries of $E$ are general. Since $E$ is assumed to be nonvanishing, we have 
$d^\pm\ge1$. 

Our simplifying assumption, which is natural in term-reversible situations, is that $S$ is also diagonal in these coordinates, so
$S=\diag\{S_1,S_2,\dots,S_d\}$ where $S_\beta\in \{\pm 1\}$, 
and that $S\neq\pm I$.
Define $B\subset\{1,\dots,d\}$ so that $S_\beta=1$ if and only if $\beta\in B$.
Then the constraint~\eqref{eq:b} reduces to the constraints
\[
b^{\beta\gamma}(Sx)=b^{\beta\gamma}(x)
\quad\text{for $\beta\in B$, $\gamma>d^+$ and for $\beta\not\in B$, $\gamma\le d^+$},
\]
and
\[
b^{\beta\gamma}(Sx)=-b^{\beta\gamma}(x)
\quad\text{for $\beta\in B$, $\gamma\le d^+$ and for $\beta\not\in B$, $\gamma > d^+$}.
\]

Fix $i\in B$, $j\not\in B$. Then an allowable choice of $b$ is obtained by setting
\[
b^{id}(X)=b^{j1}(X)=X^i
\]
and setting the remaining entries to zero.
Substituting into the sum in~\eqref{eq:correct}, we see immediately that nonzero terms require $\alpha=i$. Then the factor $b^{\alpha\gamma}(X)$ is nonzero only for $\gamma=d$. The column vector $b^\beta(X)$ has nonzero entries only for 
$\beta=1,d$. Since $E^{dd}=0$, the correction~\eqref{eq:correct} reduces to a single term 
\[
\tfrac12 E^{d1}\partial_i b^1(X) b^{id}(X)=\tfrac12 E^{d1}  X^i e_j
\]
where $e_j$ is the canonical unit vector with $1$ in the $j$'th entry.
We know that typically $E^{d1}\neq0$, so this yields a nontrivial correction to the drift term $a$ as required.

\paragraph{Acknowledgements}
This research was supported in part by the visitor program at the Sydney Mathematical Research Institute (SMRI) in early 2020. IM is very grateful for the hospitality of the University of Sydney, and in particular SMRI, at both ends of the pandemic.

We would like to thank Greg Pavliotis for his comments on detailed balance.

\end{document}